\documentclass{amsart}
\usepackage[utf8]{inputenc}
\usepackage{amsthm}
\usepackage{amsmath}
\usepackage{amssymb} 
\usepackage{comment}
\usepackage{graphicx}
\usepackage{mathrsfs}
\usepackage{fullpage}
\usepackage{colonequals}
\usepackage[all]{xy}
\usepackage[usenames,dvipsnames]{color}

\numberwithin{equation}{subsection} 
\numberwithin{figure}{subsection} 

\makeatletter
\let\c@equation\c@figure
\makeatother

\newcommand{\Q}{\mathbb{Q}}
\newcommand{\R}{\mathbb{R}}
\newcommand{\Z}{\mathbb{Z}}

\newcommand{\kbar}{{\overline{k}}}

\newcommand{\ksep}{{k^{\operatorname{sep}}}}

\newcommand{\slantsf}[1]{\textsl{\textsf{#1}}}

\newcommand{\et}{{\text{\'{e}t}}}
\newcommand{\GL}{\operatorname{GL}}
\newcommand{\op}{{\operatorname{op}}}

\DeclareMathOperator{\End}{End}

\newtheorem{theorem}[equation]{Theorem}
\newtheorem{corollary}[equation]{Corollary}
\newtheorem{lemma}[equation]{Lemma}
\newtheorem{proposition}[equation]{Proposition}

\theoremstyle{definition}
\newtheorem{remark}[equation]{Remark}

\newtheorem{question}[equation]{Question}

\theoremstyle{definition}
\newtheorem{definition}[equation]{Definition}

\newcommand{\Laur}[1]{{\langle\!\langle#1\rangle\!\rangle}}
\newcommand{\Qlr}[2]{{\mathbb{Q}_\ell^{\leqslant #2}\langle\!\langle#1\rangle\!\rangle}}

\newcommand{\Qlaur}[1]{{\mathbb{Q}_\ell\langle\!\langle#1\rangle\!\rangle}}
\newcommand{\Kr}[2]{{K^{\leqslant #2}\langle\!\langle#1\rangle\!\rangle}}
\newcommand{\Klaur}[1]{{K\langle\!\langle#1\rangle\!\rangle}}
\newcommand{\norm}[2]{{\Vert #1 \Vert_{#2}}}

\newcommand\nc{\newcommand}
\nc\on{\operatorname}

\title{Level structure, arithmetic representations, and noncommutative Siegel linearization}
\author{Borys Kadets and Daniel Litt}
\date{}
\address{Department of Mathematics\\Boyd Graduate Studies Research Center\\
University of Georgia\\
Athens, GA 30602}
\subjclass{11G10, 14F35}
\begin{document}

\maketitle

\begin{abstract}
    Let $\ell$ be a prime, $k$ a finitely generated field of characteristic different from $\ell$, and $X$ a smooth geometrically connected curve over $k$. Say a semisimple representation of $\pi_1^{\et}(X_{\bar k})$ is arithmetic if it extends to a finite index subgroup of $\pi_1^{\et}(X)$. We show that there exists an effective constant $N=N(X,\ell)$ such that any semisimple arithmetic representation of $\pi_1^{\et}(X_{\bar k})$ into $\GL_n(\overline{\mathbb{Z}_\ell})$, which is trivial mod $\ell^N$, is in fact trivial. This extends a previous result of the second author from characteristic zero to all characteristics. The proof relies on a new noncommutative version of Siegel's linearization theorem and the $\ell$-adic form of Baker's theorem on linear forms in logarithms.
\end{abstract}

\section{Introduction}
The main goal of this note is to analyze representations of arithmetic fundamental groups, motivated by questions about level structure of Abelian varieties over function fields. Our main result (Theorem \ref{main theorem}) implies that there is an absolute bound on the maximum $N$ such that an Abelian scheme over a fixed curve over a field of characteristic prime to $\ell$ has full level $\ell^N$-structure (Corollary \ref{cor:abelian-variety}). The contribution of this note is to show that this phenomenon, which was already known to hold in characteristic zero (see \cite[Theorem 1.2]{litt-inventiones} and \cite[Theorem 1.1.13]{litt-duke}) in fact holds in arbitrary characteristic. The proof in positive characteristic requires significant input from non-Archimedean dynamics and transcendence theory: in particular, a new noncommutative, non-Archimedean version of Siegel's linearization theorem (Theorem \ref{Siegel}), which may be of independent interest, and the $\ell$-adic version of Baker's theorem on linear forms in logarithms (due to Kunrui Yu \cite{yu-baker}).
\subsection{Main Results}
In preparation for the statement of the main theorem, we recall the definition of an arithmetic representation. The main interest in this notion stems from the fact that representations ``arising from geometry" are arithmetic. (See, e.g.~\cite{litt-duke} for a discussion of arithmeticity, its properties, and its consequences.)
\begin{definition}
Let $k$ be a finitely generated field, and let $X/k$ be a variety. Let $\ksep$ be a separable closure of $k$, and let $\bar x$ be a geometric point of $X$. Then a continuous representation $$\rho: \pi_1^\et(X_{\ksep}, \bar x)\to \GL_n(\overline{\mathbb{Q}_\ell})$$ is said to be \emph{semisimple arithmetic} if
\begin{enumerate}
    \item $\rho$ is semisimple, and
    \item there exists a finite separable extension $k'/k$ and a representation $$\tilde\rho: \pi_1^\et(X_{k'}, \bar x)\to \GL_n(\overline{\mathbb{Q}_\ell})$$ such that $\tilde\rho|_{\pi_1(X_{\ksep}, \bar x)}$ is conjugate to $\rho$.
\end{enumerate}
\end{definition}
Our main result about semisimple arithmetic representations is:
\begin{theorem}\label{main theorem}
Let $k$ be a finitely-generated field, and let $X/k$ be a smooth curve. Let $\ell$ be a prime not equal to the characteristic of $k$, and let $\ksep$ be a separable closure of $k$. Let $\bar x$ be a geometric point of $X$. There exists a positive constant $N=N(X,\ell)$ such that if $$\rho: \pi_1^{\et}(X_{\ksep}, \bar x)\to \GL_n(\overline{\mathbb{Z}_\ell})$$ is a continuous representation such that
\begin{enumerate}
\item $\rho\otimes \overline{\mathbb{Q}_\ell}$ is semisimple arithmetic, and
\item $\rho$ is trivial modulo $\ell^N$,
\end{enumerate}
then $\rho$ is trivial.
\end{theorem}
For a real number $N$ we say that a representation is \slantsf{trivial modulo $\ell^N$} if it is trivial modulo the ideal $$\{x\in \overline{\mathbb{Z}_\ell}\mid v_\ell(x)\geqslant N\}.$$

Theorem \ref{main theorem} immediately implies (by the Lang-N\'eron theorem, \cite[Theorem 2.1]{lang-neron}):
\begin{corollary}\label{cor:abelian-variety}
Let $X, N,\ell$ be as above, and let $\eta$ be the function field of $X_{\bar k}$. Then for any integer $M>N$, and any Abelian scheme $A/X_{\bar k}$, the following holds: if the Abelian variety $A_\eta$ has full $\ell^M$-torsion (that is,~$A_\eta[\ell^M](\eta) = A_\eta[\ell^M](\overline{\eta})$), then $A_\eta$ is isogenous to an isotrivial abelian variety over $\eta$.
\end{corollary}
\begin{remark}
In fact the conclusion of Corollary \ref{cor:abelian-variety} is equivalent to the claim that $A_\eta$ is isogenous to $\on{Tr}_{\eta/\bar k}(A_\eta)_\eta$. In characteristic zero, we may conclude that $A_\eta$ is in fact isotrivial; this is not the case in positive characteristic. See e.g.~the well-known examples due to Moret-Bailly \cite{moret1981familles}.
\end{remark}
\begin{remark}
The constant $N$ in Theorem \ref{main theorem} and Corollary \ref{cor:abelian-variety} is in principle explicit: it depends only on $\ell$ and on the natural Galois representation $$\on{Gal}(\ksep/k)\to \GL(H^1(X_{\bar k}, \mathbb{Z}_\ell)).$$ See Section \ref{sec:remarks} for related questions.
\end{remark}
\subsection{Relation to previous work}
As far as we know, this is the first result of this form in positive characteristic. In characteristic zero, Theorem \ref{main theorem} was already known by work of the second author \cite[Theorem 1.2]{litt-inventiones}. In fact more is known (\cite[Theorem 1.1.13]{litt-duke}): namely, that for fixed $X$, the constants $N(X, \ell)$ appearing in the statement of Theorem \ref{main theorem} may be taken to go to zero as $\ell\to \infty$. In positive characteristic we are unable to prove even that $N(X,\ell)$ may be bounded independent of $\ell$, as the existing bounds arising from the $\ell$-adic form of Baker's Theorem on linear forms in logarithms \emph{get worse} as $\ell\to \infty$ (see Question \ref{baker-question}). 

There is also related work of an analytic nature in characteristic zero; see e.g.~\cite{bakker-tsimerman, bakker-tsimerman2, nadel, hwang-to}.

\subsection{Outline of Proof}
We briefly sketch the idea of the proof of Theorem \ref{main theorem}, which broadly follows the strategy of \cite{litt-inventiones, litt-duke}, but replaces the use of Galois homotheties (arising from Bogomolov's $\ell$-adic open image theorem) by the use of Frobenii, and replaces the naive estimates there with more sophisticated estimates arising from our form of Siegel's theorem.

A specialization argument reduces the theorem to the case where $k$ is finite, $X$ is affine, and $X(k)$ is nonempty. Let $x\in X(k)$ be a rational point and $\bar x$ an associated geometric point; let $\pi_1^\ell(X_{\kbar}, \bar x)$ be the pro-$\ell$ completion of the geometric \'etale fundamental group of $X$. The pro-$\ell$ group ring $\mathbb{Z}_\ell\Laur{\pi_1^\ell(X_{\bar k}, \bar x)}$ is (non-canonically) isomorphic to a ring of noncommutative power series over $\mathbb{Z}_\ell$. Letting $K$ be a finite extension of $\mathbb{Q}_\ell$ the completion $K\Laur{\pi_1^\ell(X_{\bar k}, \bar x)}$ of $\mathbb{Z}_\ell\Laur{\pi_1^\ell(X_{\bar k}, \bar x)}\otimes K$ at the augmentation ideal is (non-canonically) isomorphic to a ring of noncommutative power series over $K$. For each positive real number $0<r<1$ we introduce a certain Banach sub-algebra $\Kr{\pi_1^\ell(X_{\bar k}, \bar x)}{r}$ of $K\Laur{\pi_1^\ell(X_{\bar k}, \bar x)}$, containing $\mathbb{Z}_\ell\Laur{\pi_1^\ell(X_{\bar k}, \bar x)}$, with the following property: any representation $\rho$ of $\pi_1^\ell(X_{\bar k}, \bar x)$ into $\GL_n(\overline{\mathbb{Z}_\ell})$, which is trivial modulo $\ell^N$, extends canonically to a continuous homomorphism $$\widetilde\rho: \Kr{\pi_1^\ell(X_{\bar k}, \bar x)}{r'}\to \on{Mat}_{n\times n}(\overline{\mathbb{Q}_\ell})$$ for every $r'>\ell^{-N}$.

Our main technical result is that for some finite extension $K$ of $\mathbb{Q}_\ell$, the eigenvectors of the Frobenius action on the Banach algebra $\Kr{\pi_1^\ell(X_{\bar k}, \bar x)}{r}$ have dense span once $r$ is sufficiently small, say for $0<r\leqslant r_0$. This follows from a noncommutative variant of Siegel's linearization theorem (Theorem \ref{Siegel}), whose proof is a version of Newton's method. The hypotheses of the theorem follow in our case from the Weil conjectures for curves, the $\ell$-adic form of Baker's theorem on linear forms in logarithms (due to Kunrui Yu \cite{yu-baker}), and a semisimplicity result proven in previous work of the second author \cite[Theorem 2.20]{litt-inventiones}. The introduction of these dynamical techniques is the main innovation of this work.

Now suppose we take $N$ greater than $-\log r_0/\log \ell \in \R$, and $\rho$ is as in Theorem \ref{main theorem}. Then since $r_0>\ell^{-N}$ we obtain a canonical continuous map $$\widetilde{\rho}: \Kr{\pi_1^\ell(X_{\bar k}, \bar x)}{r_0}\to\on{Mat}_{n\times n}(\overline{\mathbb{Q}_\ell})$$ extending $\rho$. The homomorphism $\widetilde{\rho}$ is equivariant for the action of some power of Frobenius by arithmeticity of $\rho\otimes \mathbb{Q}_\ell$. Now the finite-dimensionality of $\on{Mat}_{n\times n}(\overline{\mathbb{Q}_\ell})$ implies that almost all eigenvectors of the Frobenius action on $\Kr{\pi_1^\ell(X_{\bar k}, \bar x)}{r_0}$ are sent to zero by $\widetilde{\rho}$; their density implies that some power of the augmentation ideal of $\Kr{\pi_1^\ell(X_{\bar k}, \bar x)}{r_0}$ and hence of $\mathbb{Z}_\ell\Laur{\pi_1^\ell(X_{\bar k}, \bar x}$ is sent to zero. Thus the representation $\rho$ was unipotent, and any unipotent semisimple representation is trivial, completing the proof.
\subsection{Acknowledgments}
We are grateful for useful conversations with H\'el\`ene Esnault, Moritz Kerz, Samit Dasgupta, and Simion Filip. We'd also like to thank the referee for their many extremely helpful suggestions, and for catching and fixing an error in an earlier version of the paper. This material is partly based upon work supported by the NSF grant DMS-1928930 while the first author was in residence at the Mathematical Sciences Research Institute in Berkeley, California, during the Fall 2020 semester. The second author was supported by NSF grant DMS-2001196. 
\section{Notation and preliminaries}
\subsection{Notation} Throughout $X$ will be a smooth, geometrically connected curve over a field $k$, and $\ell$ will be a prime different from the characteristic of $k$. Fix an algebraic closure $\bar k$ of $k$ and a geometric point $\bar x$ of $X$. Let $K$ denote a finite extension of $\Q_\ell$. We denote by $\pi_1^\ell(X_{\bar k}, \bar x)$ the pro-$\ell$ completion of $\pi_1^{\et}(X_{\bar k}, \bar x)$. We denote by $$\mathbb{Z}_\ell\Laur{\pi_1^\ell(X_{\bar k}, \bar x)}\colonequals \varprojlim \mathbb{Z}_\ell[H]$$ the $\mathbb{Z}_\ell$-group ring of $\pi_1^\ell(X_{\bar k}, \bar x)$; the inverse limit is taken over all finite quotients of $\pi_1^\ell(X_{\bar k}, \bar x).$ Let $\mathscr{I}\subset \mathbb{Z}_\ell\Laur{\pi_1^\ell(X_{\bar k}, \bar x)}$ be the augmentation ideal. Then we denote by $$K\Laur{\pi_1^\ell(X_{\bar k}, \bar x)}\colonequals\varprojlim_n (\mathbb{Z}_\ell\Laur{\pi_1^\ell(X_{\bar k}, \bar x)}/\mathscr{I}^n\otimes K)$$ the completion of $\mathbb{Z}_\ell\Laur{\pi_1^\ell(X_{\bar k}, \bar x)}\otimes K$ at the augmentation ideal. We abuse notation and denote the augementation ideal of $K\Laur{\pi_1^\ell(X_{\bar k}, \bar x)}$ by $\mathscr{I}$ as well.

The key ingredient of our argument is an analysis of the Frobenius action on $K\Laur{\pi_1^{\ell}(X_{\bar k}, \bar x)}$ and related algebras. For $X$ an affine curve, this ring is non-canonically isomorphic to a noncommutative power series ring in $n$ variables $x_1,\cdots, x_n$, where $x_i=\gamma_i-1$, for $\{\gamma_1, \cdots, \gamma_n\}$ a minimal set of topological generators of the free pro-$\ell$ group $\pi_1^\ell(X_{\bar k}, \bar x)$. Dealing with coefficients of such noncommutative power series brings some notational difficulties. In this section we introduce some conventions that somewhat simplify the notation.

We write $\vec{x}$ as a shorthand for $x_1, \cdots, x_n$, so that if $X$ is affine, $K\Laur{\pi_1^{\ell}(X_{\bar k}, \bar x)}=K\Laur{\vec{x}}$, the non-commutative power series ring in $x_1, \cdots, x_n$. If $I=\{i_1, \cdots, i_m\}$ is a finite word in the alphabet $\{1, \cdots, n\}$, we write $x^I\colonequals x_{i_1}\cdots x_{i_m}$. For an element $f \in K\Laur{\vec{x}}$ we write $f=\sum a_{I}x^I$ suppressing that summation is to be taken over all finite words in $\{1, \cdots, n\}$. If $I$ is a finite word, then $|I|$ denotes its length. The \slantsf{weight} of the monomial $x^{I}$ is $|I|$.

\begin{definition}
A power series $f \in K\Laur{\vec{x}}$ \slantsf{converges on a disk of radius $0<r<1$} if $$\lim_{|I|\to\infty} |a_I| r^{|I|} = 0.$$ In this case we define the $r$-norm of $f$ to be $\norm{f}{r}\colonequals \sup_I |a_I| r^{|I|}$. The set of convergent power series is denoted by $\Kr{\vec{x}}{r}$.
\end{definition}
We think of these convergent power series as functions on a closed noncommutative polydisk. From the point of view of this metaphor the Frobenius action is an automorphism of the polydisk fixing zero. We think of our main result as saying that this automorphism is conjugate to a linear map in a neighborhood of zero; the analogous  result in holomorphic dynamics on  a complex polydisk is Siegel's linearization theorem. 
\begin{proposition}
The pair $\Kr{\vec{x}}{r}, \norm{\cdot}{r}$ is a Banach algebra.
\end{proposition}
\begin{proof}
The fact that $\norm{\cdot}{r}$ is a non-Archimedean norm follows from the fact that $|\cdot|$ is a norm on $K$. For completeness, observe that a Cauchy sequence has to converge coefficient-wise, and the resulting limit is also the limit in the norm. 
\end{proof}
We similarly define $\Kr{\pi_1^\ell(X_{\bar k}, \bar x)}{r}$ to be the subring of $K\Laur{\pi_1^\ell(X_{\bar k}, \bar x)}$ consisting of those elements $f$ that satisfy $$\lim_{n\to \infty} \left(r^n\inf\{\ell^{s}\mid \ell^s\cdot f\in \mathscr{O}_K\otimes\mathbb{Z}_\ell\Laur{\pi_1^\ell(X_{\bar k}, \bar x)} \bmod \mathscr{I}^{n+1}\}\right)=0,$$
and define $$\norm{f}{r}\colonequals\sup_n \left(r^n\inf\{\ell^{s}\mid \ell^s\cdot f\in  \mathscr{O}_K\otimes\mathbb{Z}_\ell\Laur{\pi_1^\ell(X_{\bar k}, \bar x)} \bmod \mathscr{I}^{n+1}\}\right).$$  Under a choice of isomorphism $K\Laur{\pi_1^\ell(X_{\bar k}, \bar x)}\simeq K\Laur{\vec x}$ arising from a minimal set of topological generators of $\pi_1^\ell(X_{\bar k}, \bar x)$ in the case $X$ is affine, these two definitions are compatible, i.e.~the isomorphism induces an isomorphism $\Kr{\pi_1^\ell(X_{\bar k}, \bar x)}{r}\simeq \Kr{\vec x}{r}$. Moreover these constructions are functorial: any endomorphism of $\pi_1^\ell(X_{\bar k}, \bar x)$ induces an endomorphism of $K^{\leqslant r}\Laur{\pi_1^\ell(X_{\bar k}, \bar x)}$ and hence of $\Kr{\vec{x}}{r}$ when $X$ is affine.

We denote by $\mathscr{I}_r$ the ideal of power series with zero constant term in $\Kr{\vec{x}}{r}$; if the radius $r$ is clear from context we write simply $\mathscr{I}$. To emphasize the analogy with analysis we write $f=g+O(x^n)$ to mean $f-g\in \mathscr{I}^n$.

\subsection{Basic properties} We will study the dynamics of noncommutative power series on a polydisk. We use $\End^{\op}K\Laur{\vec{x}}$ to denote the monoid of $n$-tuples of power series $\vec{f}=(f_1, \dots, f_n) \in K\Laur{\vec{x}}^n$ with no constant term and with composition given by $\vec{f} \circ \vec{g} = (f_1(\vec{g}), \dots, f_n(\vec{g}))$; elements of $\End^{\op}K\Laur{\vec{x}}$ define endomorphisms of $K\Laur{\vec{x}}$ that send $x_i$ to $f_i$. Note that the order of composition on $\End^{\op}\Klaur{\vec{x}}$ is opposite of the natural composition on the endomorphisms, hence the notation. Similarly we use $\End^{\op}\Kr{\vec x}{r}$ for the $n$-tuples $\vec{f}=(f_1, \dots, f_n)$ of power series in $\Kr{\vec{x}}{r}$ with $\norm{f_i}{r} \leqslant r$ for all $i$.  The norm of $\vec{f}\in \End^{\op}\Kr{\vec x}{r}$ is defined by $\norm{\vec{f}}{r}\colonequals \max_i \norm{f_i}{r}$. We think of an element $\vec{f}\in \End^{\op}\Kr{\vec x}{r}$ as a holomorphic map from the noncommutative polydisk $\Kr{\vec{x}}{r}$ to itself, fixing the origin.

The algebras $\Kr{\vec{x}}{r}$ will be used to study representations trivial modulo $\ell^N$ with the help of the following lemma. In the lemma, the $\ell$-adic valuation $v_\ell(A)$ of a matrix $A$ is the minimal valuation of its entries, and the $\ell$-adic norm is defined by $|A|=\ell^{-v_\ell(A)}$.
\begin{lemma}\label{extends-to-Kr}
Suppose $\rho: \Z_\ell \Laur{\vec{x}} \to \mathrm{Mat}_{n\times n}(\overline{\Z_\ell})$ is a continuous representation such that $|\rho(x_i)| \leqslant C < 1$ for all $i$. Then for any $r>C$ the representation $\rho$ extends to a continuous representation $\rho_r: \Qlr{\vec{x}}{r} \to \mathrm{Mat}_{n\times n}(\overline{\Q_\ell}).$
\end{lemma}
\begin{proof}
Suppose $f=\sum a_Ix^I \in \Qlr{\vec{x}}{r}$ is a power series. Then $|a_I \rho(x)^I|\leq\norm{f}{r}r^{-|I|}C^{|I|}$. Since $r^{-1}C<1$, the series $f(\rho(\vec{x}))$ converges, and if $c\colonequals \max_i r^{-i}C^{i}$, then $|f(\rho(\vec{x}))| \leqslant c\norm{f}{r}$ by the ultrametric inequality. Therefore the representation $\rho_r: \Qlr{\vec{x}}{r} \to \mathrm{Mat}_{n\times n}(\overline{\Q_\ell})$, $\rho_r:f \mapsto f(\rho(\vec{x}))$ is a well-defined continuous extension of $\rho$.
\end{proof}
\begin{lemma}\label{lemma:dense-monomials}
The monomials of degree $\geqslant m$ have dense span in $\mathscr{I}_r^m\subset  \Kr{\vec{x}}{r}.$
\end{lemma}
\begin{proof}
Let $$f(\vec x)=\sum_{|I|\geqslant m} a_I x^I$$ be an element of $\mathscr{I}_r^n$. As $|a_I|r^I\to 0$ as $r\to \infty$, we have that, defining $$f_M:= \sum_{|I|\geqslant M}a_I x^I,$$ $f_M\to 0$ as $M\to \infty$. Hence $f=\lim_{M\to \infty} f-f_M.$ Now we have written $f$ as a limit of polynomials in the $x_i$, as desired.
\end{proof}
\section{Siegel's linearization theorem}
The goal of this section is to show that the Frobenius action on the convergent group ring $\Kr{\pi_1^\ell(X, x)}{r} $ can be diagonalized once $r$ is sufficiently small. To this end we prove a version of Siegel's linearization theorem for $\ell$-adic noncommutative multivariate power series. Our argument is analogous to the one used in \cite[Section 4]{herman-yoccoz}, which itself is an $\ell$-adic version of a classical argument of R\"{u}ssmann, as generalized by Zehnder \cite{riissmann, zehnder}. That said, we require these results without the ``non-resonance" condition they impose. (In our application, this non-resonance condition is replaced by the use of the semisimplicity of the Frobenius action on the algebra $K\Laur{\pi_1^\ell(X, x)}$.) We start by establishing some simple lemmas on composition of power series.
\subsection{Some lemmas}
The following lemma establishes basic analytic properties of $\End^\op \Kr{\vec{x}}{r}$.
\begin{lemma}\label{fixed-Endop}
\noindent
\begin{enumerate}
    \item The set $\End^\op \Kr{\vec{x}}{r}$ is a complete abelian group under the norm $\norm{\cdot}{r}$ with respect to addition;
    \item The group $\End^\op \Kr{\vec{x}}{r} \subset \End^\op \Klaur{\vec{x}}$ is closed under composition.
    \item \label{composition-inequality} Composition on $\End^\op \Kr{\vec{x}}{r}$ is continuous and satisfies the inequalities
  \begin{align*}
      \norm{\vec{f}\circ\vec{g}-\vec{f}\circ \vec{h}}{r} &\leqslant \frac{1}{r}\norm{\vec{f}}{r}\norm{\vec{g}-\vec{h}}{r}, \\
      \norm{\vec{g}\circ \vec{f} - \vec{h}\circ \vec{f}}{r} &\leqslant \frac{1}{r}\norm{\vec{f}}{r}\norm{\vec{g}-\vec{h}}{r}
  \end{align*}
   % \[\norm{\vec{f}\circ\vec{g}-\vec{f}\circ \vec{h}}{r} \leqslant \frac{1}{r}\norm{\vec{f}}{r}\norm{\vec{g}-\vec{h}}{r}, \text{and}\]
  %  \[\norm{\vec{g}\circ \vec{f} - \vec{h}\circ \vec{f}}{r} \leqslant \frac{1}{r}\norm{\vec{f}}{r}\norm{\vec{g}-\vec{h}}{r}\]
    for all $f,g,h \in \End^\op \Kr{\vec{x}}{r}.$
    \item \label{schwarz} For $0<r_1<r_2<1$ we have $\Kr{\vec{x}}{r_2}\subset \Kr{\vec{x}}{r_1}\subset K\Laur{\vec{x}}$, and the inclusion is continuous. Given $\vec f\in \Kr{\vec x}{r_2},$ we have $$\norm{\vec f}{r_1}\leqslant \frac{r_1}{r_2}\norm{\vec f}{r_2}.$$
\end{enumerate}
 
\end{lemma}
\begin{proof}
\begin{enumerate}
    \item This follows from the completeness of $K$; if $\{\vec f_j\}$ is a Cauchy sequence in $\End^\op \Kr{\vec{x}}{r}$, then the coefficients of $x^I$ in each component of $\vec f_j$ form a Cauchy sequence themselves.
    \item Given $\vec f, \vec g\in \End^\op \Kr{\vec{x}}{r}$, we wish to show that $\norm{f_i(\vec g)}{r}\leqslant r$ for all $i$. Writing $f_i=\sum a_I^ix^I$, $g_j=\sum b_{I,j}x^I$, the coefficient of $x^I$ in $f_i(\vec g)$ is an integer linear combination of $a_Jb_{I_1, j_1}\cdots b_{I_{|J|}, j_{|J|}}$ where $\sum_j I_j=I$. Hence $$\norm{f_i(\vec g)}{r}\leqslant \sup_{I, J, I_1, \cdots, I_{|J|}}|a_Jb_{I_1, j_1}\cdots b_{I_{|J|}, j_{|J|}}|r^{|I|}\leqslant \sup_J |a_J|r^{|J|}\leqslant r$$ as desired.
    \item We prove the first inequality; the second is the special case where we apply the first inequality to $\norm{(\vec g-\vec h)\circ \vec f - (\vec g-\vec h)\circ \vec 0}{r}$. Setting $f_i=\sum a_Ix^I$, $g_j=\sum b_{I,j}x^I$, $h_k=\sum c_{I,k}x^I$, we have that the coefficient of $x^I$ in $f_i(\vec g)-f_i(\vec h)$ is an integer linear combination of $a_J(b_{I_1, j_1}\cdots b_{I_{|J|}, j_{|J|}}-c_{I_1, j_1}\cdots c_{I_{|J|}, j_{|J|}})$ with $\sum_s I_s=I$. Using the telescoping sum 
    \begin{align*}
       b_{I_1, j_1}\cdots b_{I_{|J|}, j_{|J|}}-c_{I_1, j_1}\cdots c_{I_{|J|}, j_{|J|}} &= \sum_{s=1}^J b_{I_1, j_1}\cdots b_{I_{s-1}, j_{s-1}}(b_{I_{s}}, j_{s}-c_{I_s, j_s})c_{I_{s+1}, j_{s+1}}\cdots c_{I_{|J|}, j_{|J|}}
    \end{align*}
    we have 
    \begin{align*}
        |b_{I_1, j_1}\cdots b_{I_{|J|}, j_{|J|}}-c_{I_1, j_1}\cdots c_{I_{|J|}, j_{|J|}}|r^{|I|} &\leqslant \max_{s=1, \cdots, |J|}|b_{I_1, j_1}\cdots b_{I_{s-1}, j_{s-1}}(b_{I_{s}}, j_{s}-c_{I_s, j_s})c_{I_{s+1}, j_{s+1}}\cdots c_{I_{|J|}, j_{|J|}}|r^{|I|} \\
        &\leqslant \max_{s=1, \cdots, |J|} \left(\prod_{t=1}^{s-1}|b_{I_t, j_t}|r^{|I_t|}\right)\cdot |b_{I_s, j_s}-c_{I_s, j_s}|r^{|I_s|}\cdot\left( \prod_{t=s+1}^{|J|}|b_{I_t, j_t}|r^{|I_t|}\right)\\
        &\leqslant \left(\max_{s=1, \cdots, |J|}|b_{I_s, j_s}-c_{I_s, j_s}|r^{|I_s|}\right)r^{|J|-1}\\
        &\leqslant \norm{\vec{g}-\vec{h}}{r}r^{|J|-1}.
    \end{align*}
    But then 
    \begin{align*}
        |a_J(b_{I_1, j_1}\cdots b_{I_{|J|}, j_{|J|}}-c_{I_1, j_1}\cdots c_{I_{|J|}, j_{|J|}})|r^{|I|} &\leqslant |a_J|\cdot\norm{\vec{g}-\vec{h}}{r}r^{|J|-1}\\
        &\leqslant \frac{1}{r}\norm{\vec{f}}{r}\norm{\vec{g}-\vec{h}}{r}
    \end{align*}
    as desired.
    \item The inequality $\norm{\vec f}{r_1}\leqslant \frac{r_1}{r_2}\norm{\vec f}{r_2}$ immediately implies that there is a continuous inclusion $\Kr{\vec{x}}{r_2}\subset \Kr{\vec{x}}{r_1}$, so it suffices to verify the inequality. Writing $f_i=\sum a_Ix^I$, it is enough to show that $$\sup_I |a_I|r_1^{|I|}\leqslant \frac{r_1}{r_2}\sup_I |a_I|r_2^{|I|},$$ where the supremum is taken over $I$ with $|I|>0$. But for each $I$ with $|I|>0$, we have $$|a_I|r_1^{|I|}\leqslant |a_I|r_1r_2^{|I|-1},$$ which gives the claim.
\end{enumerate}
\end{proof}
As a special case of Lemma \ref{fixed-Endop}(\ref{composition-inequality}) above, we have:
\begin{lemma}\label{Taylor}
For any $\vec{f}, \vec{\varepsilon} \in \End^{\op}\Kr{\vec{x}}{r}$ with $\norm{ \vec{\varepsilon}}{r} < 1$, and for any diagonal matrix $A=\mathrm{diag}(\lambda_i)\in \mathrm{Mat}_{n \times n}(K)$, $|\lambda_i|\leqslant 1$ the following estimate holds: \[\norm{\vec{f}(A\vec{x}+\vec{\varepsilon})-\vec{f}(A\vec{x})}{r} < \frac{1}{r} \norm{\vec{f}}{r}\norm{\vec{\varepsilon}}{r}.\]
\end{lemma}
%\begin{proof}
%The coefficient of a monomial of weight $N$ in $\vec{f}(A\vec{x}+\vec{\varepsilon})-\vec{f}(A\vec{x})$ is an integer linear combination of expressions of the form $\lambda a_{i} b_{i_1} \dots b_{i_m}$ where $\lambda$ is (possibly empty) product of $\lambda_k$, $a_i$ is a coefficient of $f$ of weight $i$, $b_{i_j}$ is a coefficient of $\vec{\varepsilon}$ of weight $i_j$, $m>0$, and $i+\sum_j i_j=N$. We have 
%\[|\lambda a_i b_{i_1}\dots b_{i_m}| \leqslant \norm{\vec{f}}{r}r^{-i}\norm{\vec{\varepsilon}}{r}r^{-i_1}\dots \norm{\vec{\varepsilon}}{r}r^{-i_m}=\norm{\vec{f}}{r}\Vert\vec{\varepsilon}\Vert_r^m r^{-(i+i_1+\dots+i_m)}\leqslant \norm{\vec{f}}{r}\Vert\vec{\varepsilon}\Vert_r r^{-N}.\]
%\end{proof}
We will require a criterion for invertibility of certain elements of $\End^\op \Kr{\vec x}{r}$:
\begin{lemma}\label{inversion}
Suppose $\vec{\psi} \in \End^{\op}\Kr{\vec{x}}{r}$ is of the form $\vec{\psi}=\vec{x}+\hat{\psi}$, $\hat{\psi}=O(x^2)$ and $\norm{\hat{\psi}}{r}<\varepsilon<r$. 
Then $\vec{\psi}$ admits a two-sided compositional inverse $\vec{g}=\vec{x}+\hat{g}$ and $\norm{\hat{g}}{r}<\varepsilon.$
\end{lemma}
\begin{proof}
We first find the left inverse. Set $\vec g_L=\vec x -\hat \psi + \hat \psi^{\circ 2}-\hat \psi^{\circ 3}+\cdots$. The sum converges as $\norm{\hat\psi}{r}<\varepsilon<r$, and hence by \ref{fixed-Endop}(\ref{composition-inequality}), $\norm{\hat\psi^{\circ n}}{r}<\varepsilon^n/r^{n-1}$, by induction on $n$. Thus $\norm{\hat\psi^{\circ n}}{r}\to 0$ as $n\to\infty$. Moreover $\hat g_L:=\vec g_L-\vec x$ satisfies $\norm{\hat g_L}{r}<\varepsilon$ by the ultrametric inequality. Thus it suffices to show $\vec g_L$ is inverse to $\vec \psi$.

To see that $\vec g_L$ is left inverse to $\vec\psi$, we simply evaluate $\vec g_L\circ \vec\psi$; the sum telescopes.

The previous paragraphs shows that any tuple of power series $\vec{\psi}$ satisfying the conditions of the lemma admits a left inverse $\vec{g}_{L}$. Applying this to $\vec{g}_L$ shows that $\vec{g}_L$ admits a left inverse. Since $\vec{g}_L$ also has a right inverse, $\vec{\psi}$ is a two-sided inverse of $\vec{g}_L$, and $\vec{g}_L$ is a two-sided inverse of $\vec{\psi}$. 
\end{proof}
\begin{definition}\label{def-semisimple}
An element $\vec{f} \in \End^{\op}\Kr{\vec{x}}{r}$ is \slantsf{semisimple} if the operator $P \mapsto P \circ \vec{f}$  on $\Kr{\vec{x}}{r}/\mathscr{I}^m$ is semisimple for all $m$.
\end{definition}
\begin{remark}
The property of being semisimple is preserved under conjugation. For all $r'<r$ an element $\vec f \in \End^\op \Kr{\vec{x}}{r}$ is semisimple as an element of $\End^\op \Kr{\vec{x}}{r}$ if and only if it is semisimple as an element of $\End^\op \Kr{\vec{x}}{r'}$.

%For any $r'<r$ an element of $\End \Kr{\vec{x}}{r}$ is semisimple if and only if it is semisimple as an element of $\End \Kr{\vec{x}}{r'}$.
\end{remark}
\begin{lemma}\label{off-diagonal}
Suppose $\vec{f} \in \End^{\op}\Kr{\vec{x}}{r}$ is semisimple, $\vec{f}=A\vec{x}+O(x^2)$, and $A$ is a diagonal matrix with coefficients $\vec{\lambda}=(\lambda_1, \dots, \lambda_n)$. If $I$ is a word and $j$ is an index such that $\vec{\lambda}^I=\lambda_j$, then the coefficient of $x^I$ in $f_j$ is zero.
\end{lemma}
\begin{proof}
The operator $F \mapsto F \circ \vec{f}$ on $K\Laur{\vec{x}}/\mathscr{I}^{|I|+1}$ is upper triangular in the monomial basis. Therefore, since the diagonal entries corresponding to the coefficients of $x^I$ and $x_j$ are equal and the operator is semisimple, the corresponding off-diagonal coefficient $a_I^j$ is zero.
\end{proof}
\begin{definition}\label{def:l-siegel}
A tuple of numbers $\lambda_1, \dots, \lambda_n \in \overline{\Q_\ell}$ is said to satisfy \emph{$\ell$-Siegel's condition with parameters $c, \mu>0$} if the following condition holds: 
for any tuple of nonnegative integers $i_1, \dots, i_n$ with $i_1+\dots+i_n=N$ and any index $j$ such that $\lambda_1^{i_1}\lambda_2^{i_2}\dots\lambda_n^{i_n} \neq \lambda_j$, the following inequality holds
\[\vert\lambda_1^{i_1}\lambda_2^{i_2}\dots\lambda_n^{i_n} - \lambda_j \vert \geqslant c(N/2)^{-\mu}.\]
\end{definition}
\begin{remark}
Replacing $\lambda_j$ with $1$ and $N/2$ with $N$ leads to an equivalent definition with different constants; the form given in Definition \ref{def:l-siegel} is more convenient for our applications.
\end{remark}

Siegel's condition holds for algebraic $\lambda_i$, as the following proposition shows.

\begin{proposition}[Linear forms in logarithms, \cite{yu-baker}]\label{prop:baker}
Suppose $\lambda_1, \dots, \lambda_n\in \overline{\mathbb{Q}_\ell}$ are algebraic numbers. Then there exist constants $c, \mu >0$ such that for any integers $i_1,\dots , i_n,j$ such that $\lambda_1^{i_1}\dots \lambda_n^{i_n} \neq \lambda_j$ the following inequality holds:
\[\left| \lambda_1^{i_1}\dots\lambda_n^{i_n}-\lambda_j \right| \geqslant c (|i_1|+\dots+|i_n|)^{-\mu}.\]
\end{proposition}
\begin{proof}
After replacing the tuple $\lambda_1, \dots, \lambda_n$ with $\lambda_1, \dots, \lambda_n, \lambda_1^{-1}, \dots, \lambda_n^{-1}$ it suffices to prove the inequality for positive integers $i_j$. Also, by changing $c$ to a different constant it suffices to show that $\left| \lambda_1^{i_1}\dots\lambda_n^{i_n}-1 \right| \geqslant c (|i_1|+\dots+|i_n|)^{-\mu}.$ This is proved in \cite[Theorem 1\textquotesingle]{yu-baker}.
\end{proof}
Our goal is to conjugate a (semisimple vector of) power series $\vec{f}=A\vec{x}+O(x^2)$ to its linear part $A\vec{x}$. We do so iteratively, on every step conjugating $f$ by some other power series so that the norm of $\vec{f}-A\vec{x}$ on a slightly smaller disk becomes much smaller than before. The key to the inductive step is Lemma \ref{Siegel iteration}, which itself uses the following simple estimate.
\begin{lemma}\label{calculus1}
Suppose $\eta \in (0,1)$ and $\mu>0$ are real numbers. Then 
\[\sup_{i \in \Z_{\geqslant 0}} (1-\eta)^ii^\mu \leqslant \left(\frac{\eta}{7\mu}\right)^{-\mu}.\]
\end{lemma}
\begin{proof}
For every positive $x$ we have \[e^{\frac{\eta x}{7\mu}} \geqslant \frac{\eta x}{7\mu}.\]

By convexity of $\log$ on $(1,2)$ we have $\log(1+\eta) \geqslant \frac{\eta}{7}$. Therefore 
\[e^{\frac{\log(1+\eta)}{\mu}x} \geqslant e^{\frac{\eta x}{7\mu}} \geqslant \frac{\eta x}{7\mu},\] Since $1/(1+\eta) > 1-\eta$ inverting both sides we get
\[\frac{7\mu}{\eta x} \geqslant (1-\eta)^{x/\mu}, \text{or}\]
\[x (1-\eta)^{x/\mu} \leqslant \frac{7\mu}{\eta}. \]
Raising both sides to the power $\mu$ gives
\[x^\mu (1-\eta)^{x} \leqslant \left(\frac{\eta}{7\mu}\right)^{-\mu}\]
\end{proof}
\begin{lemma}\label{Siegel iteration}
Given $\delta>0$, suppose $\vec{f}\in \End^{\op}\Kr{\vec{x}}{r}$ is such that $\vec{f}=A\vec{x}+\hat{f}(x), \hat{f}=O(x^2)$ and $\norm{ {\hat f} }{r} < \delta$. Suppose that $\vec{f}$ is semisimple, $A$ is a diagonal matrix, and that the eigenvalues $\lambda_1, \dots, \lambda_n$ of $A$ satisfy $\ell$-Siegel's condition with parameters $c, \mu$, $\mu>2/7$. Suppose also $|\lambda_i| \leqslant 1$. Suppose $\eta\in(0,1)$ satisfies $c^{-1}(7\mu)^\mu \eta^{-\mu}\delta < r$. Then there exists an invertible endomorphism $\vec{\psi} \in \End^{\op}\Kr{\vec x}{r(1-\eta)}$, with $\vec{\psi}=\vec{x}+O(x^2)$, such that the following inequalities hold:
\[\norm{\vec{\psi} - \vec{x}}{r(1-\eta)} \leqslant c^{-1}\delta(1-\eta)\left(\frac{\eta}{7\mu}\right)^{-\mu}<r(1-\eta)\]
\[\norm{ \vec{\psi}^{-1} \circ \vec{f} \circ \vec{\psi} - A\vec{x}}{r(1-\eta)} \leqslant  c^{-1}(7\mu)^{\mu}\frac{ \delta^2 }{\eta^{\mu}r}\]  
\end{lemma}
\begin{proof}
Let $\hat{\psi}$ be the solution to the equation $\hat{\psi}(A\vec{x})-A \hat{\psi}(\vec{x})=\hat{f}(\vec{x})$ with $\hat\psi=O(x^2)$. This equation can be solved as $A$ is diagonal: if $a_I^j$ is the coefficient of $x^{I}$ in $f_j$, then the coefficient of $x^I$ in $\psi_j$ is $a_I^j/(\lambda^I-\lambda_j)$ if $\lambda^I-\lambda_j$ is nonzero, and it is zero otherwise. (Note that $a_I^j=0$ whenever $\lambda^I=\lambda_j$ by Lemma \ref{off-diagonal}.) To check that $\hat{\psi}$ converges on the disk of radius $r(1-\eta)$ we use the Siegel condition:
\[\left|\frac{a_{I}^j}{\lambda^I - \lambda_j}\right| r^{|I|}(1-\eta)^{|I|}\leqslant \delta (1-\eta)^{|I|}|\lambda^{|I|}-\lambda_j|^{-1} \overset{\text{[by Siegel's property]}}{\leqslant} c^{-1}\delta (1-\eta)^{|I|}(|I|/2)^{\mu} \underset{|I| \to \infty}{\longrightarrow} 0.\]
Similarly, we directly estimate the norm of $\hat{\psi}$ using the Siegel condition as follows:

\[ \Vert \hat{\psi} \Vert_{r(1-\eta)}= \sup_{I,j} \left\vert \frac{a_I^j}{\lambda^I - \lambda_j}\right\vert r^{|I|}(1-\eta)^{|I|}  \leqslant \delta \sup_{I,j} (1-\eta)^{|I|}\left\vert\lambda^I - \lambda_j \right \vert^{-1} \overset{\text{[by Siegel's property]}}{\leqslant}  \]
\[\delta \sup_i c^{-1}(1-\eta)^i (i/2)^\mu \leqslant c^{-1}\delta \max(\sup_{i\geqslant 2} (1-\eta)^i (i-1)^{\mu}, (1-\eta)2^{-\mu}) \overset{[\text{Index shift}]}=\]\[ 
 c^{-1}\delta \max((1-\eta)\sup_{i\geqslant 1} (1-\eta)^i i^{\mu}, (1-\eta)2^{-\mu})\overset{[\text{Lemma }  \ref{calculus1}]}{\leqslant} c^{-1}\delta (1-\eta)\left(\frac{\eta}{7\mu}\right)^{-\mu}.\]

In the last inequality we have used the condition $\mu>2/7$ to resolve the maximum.
Note that by our choice of $\eta$, we thus have $\norm{\hat\psi}{r(1-\eta)}<r(1-\eta)$, and so $\hat\psi$ is in $\End^\op \Kr{\vec x}{r(1-\eta)}.$

We now show that the function $\vec{\psi}=\vec{x}+\hat{\psi}$ satisfies the conditions of the lemma. Note that $\vec{\psi}$ is invertible by Lemma \ref{inversion}, again as $\norm{\hat\psi}{r(1-\eta)}<r(1-\eta)$. We need to estimate the norm $\Vert \vec{\psi}^{-1}\circ \vec{f} \circ \vec{\psi} - A\vec{x} \Vert_{r(1-\eta)}$. Let $\vec{g}$ denote the function $\vec{\psi}^{-1}\circ \vec{f} \circ \vec{\psi}$, and write $\vec{g}=A\vec{x}+\hat{g}$. We now use the functional equation for $\hat{\psi}$ to derive an equation for $\hat{g}$:
\[\vec{\psi}(\vec{g}(\vec{x}))=A\vec{\psi}(\vec{x})+\hat{f}(\vec{\psi}(\vec{x}))\]
\[\hat{g}(\vec{x})+\hat{\psi}(A\vec{x} + \hat{g}(\vec{x}))= A \hat{\psi}(\vec{x}) + \hat{f}(\vec{x}+\hat{\psi}(\vec{x}))\]
\[\hat{g}(\vec{x})+\hat{\psi}(A\vec{x} + \hat{g}(\vec{x}))= \hat{\psi}(A\vec{x})-\hat{f}(\vec{x}) + \hat{f}(\vec{x}+\vec{\psi}(\vec{x}))\]
\[\hat{g}(\vec{x})=\left[\hat{\psi}(A\vec{x})-\hat{\psi}(A\vec{x}+\hat{g}(\vec{x}))\right]+\left[\hat{f}(\vec{x}+\hat{\psi}(\vec{x}))-\hat{f}(\vec{x})\right].\]
The right hand side of the last equation is visibly ``small''; if $G=\Vert \hat{g}\Vert_{r(1-\eta)}$, then applying Lemma \ref{Taylor} and the ultrametric inequality we get:
\[G \leqslant \frac{1}{r(1-\eta)}\max\{\Vert \hat{\psi} \Vert_{r(1-\eta)}G, \Vert \hat{f}\Vert_{r(1-\eta)} \Vert \hat{\psi}\Vert_{r(1-\eta)}\}\leqslant \frac{1}{r(1-\eta)}\Vert \hat{f}\Vert_{r(1-\eta)} \Vert \hat{\psi}\Vert_{r(1-\eta)}.\]
Here the last inequality holds because $G\leqslant \frac{1}{r(1-\eta)}\Vert \hat{\psi} \Vert_{r(1-\eta)}G$ is impossible, unless $G=0$, as $\frac{1}{r(1-\eta)}\Vert \hat{\psi} \Vert_{r(1-\eta)}<1;$ and if $G=0$, the inequality holds in any case.
Using the estimates $\Vert\hat{\psi}\Vert_{r(1-\eta)} \leqslant c^{-1}(7\mu)^{\mu} \delta (1-\eta)\eta^{-\mu}$ and $\Vert f \Vert_{r(1-\eta)} \leqslant \Vert f \Vert_{r}< \delta$ we get:
\[G \leqslant \frac{c^{-1}(7\mu)^{\mu} \delta^2 \eta^{-\mu}}{r}.\]\qedhere
\end{proof}

To apply Lemma \ref{Siegel iteration} we will need the following simple estimate. \
\begin{lemma}\label{calculus2}
Let $u\in (0,0.5), \alpha>1$ be real numbers. Then
\[\prod_{n=0}^\infty \left(1-\frac{u}{\alpha^n}\right) > e^{-\frac{\alpha}{\alpha - 1}}\]
\end{lemma} 
\begin{proof}
For $x \in (0,0.5)$ we have $\log (1-x) > -2x$. Taking the logarithm of the product and applying this inequality to every term gives
\[\sum_{n=0}^{\infty} \log \left(1-\frac{u}{\alpha^n}\right) > -2u\sum_{n=0}^{\infty}\left(\frac{1}{\alpha^n}\right) = -2u\frac{\alpha}{\alpha-1}>-\frac{\alpha}{\alpha-1}.\]
\end{proof}
\subsection{The linearization theorem} We are now ready to prove the main result of this section.
\begin{theorem}[noncommutative, non-archimedean Siegel linearization]\label{Siegel}
Suppose $\vec{f}=A\vec{x}+\hat{f}(x)$ is an element of $\End^{\op}\Kr{\vec{x}}{r}$. Suppose that $\vec{f}$ is semisimple, $A$ is a diagonal matrix, and that the eigenvalues $\lambda_1, \dots, \lambda_n$ of $A$ satisfy Siegel's condition with parameters $c, \mu$. Suppose $|\lambda_i| \leqslant 1.$ Then there exists a radius $r_\infty<r<1$ and an invertible function $\vec{\psi} \in \End^\op\Kr{\vec{x}}{r_\infty}$, such that $\vec{\psi}^{-1}\circ \vec{f} \circ \vec{\psi} = A\vec{x}$. 
\end{theorem}
\begin{proof}
The idea is to apply Lemma \ref{Siegel iteration} iteratively, at every step conjugating $\vec{f}$ closer and closer to its linear part. At each step we are given a conjugate $\vec{f_n}$ of $\vec{f}$ on a disc of radius $r_n$ with $\norm{\vec{f}_n - A\vec x}{r_n}=\delta_n<1$, and we have to choose a suitable value of $\eta_n$ to apply Lemma \ref{Siegel iteration} and obtain a new conjugate $\vec f_{n+1}$ on the disk of radius $r_{n+1}\colonequals r_n(1-\eta_n)$. We make these choices so that the distance between $\vec f_n$ and its linear part tends to zero with $n$, and the radii of the relevant disks remains bounded away from zero. On each step we make the choice $\eta_n\approx \delta_n^{1/(\mu+1)}$, then the numerics of Lemma \ref{Siegel iteration} show that $\delta_{n+1} \lessapprox \delta_n^2/\eta_n^\mu \approx \delta_n^{1+\frac{1}{\mu + 1}}$, so that the sequence $\delta_n$ converges to zero super-exponentially. Since $\eta_n\approx \delta_n^{1/(\mu+1)}$, the sequence $\eta_n$ will converge to zero super-exponentially as well, and so $r \prod_n(1-\eta_n)>0$; by taking limits we will show that $\vec{f}$ is conjugate to $A\vec{x}$ on the disk of radius $r \prod_n(1-\eta_n)$.

After increasing $\mu$ if necessary we can assume $\mu>2/7$ (this is a condition necessary to apply Lemma \ref{Siegel iteration}). Before starting the iterative process, we need to (possibly) shrink the disk to make the norm of $\vec f-A\vec x$ small enough so that on every step of the iteration Lemma \ref{Siegel iteration} can be applied. 

Let $\delta$ denote $\norm{\vec{f}-A\vec x}{r}$. We choose a large positive constant $B$ (to be determined later) and let $r_1 \colonequals \frac{r}{B}$. Then since $\vec{f}-A\vec x$ has no linear term, the norm $\delta_1\colonequals \norm{\vec{f}-A\vec x}{r_1}$ satisfies $\delta_1 \leqslant \frac{\delta}{B^2}$. We now choose $B$ large such that the following inequality holds 
\begin{equation}\label{Bparameter}
c^{-1}(7\mu)^\mu\delta_1 3^{\mu} <\frac{1}{2}r_1e^{-\frac{2^{1/(\mu+1)}}{2^{1/(\mu+1)}-1}},
\end{equation}
which is possible as $\delta_1$ scales as $O(1/B^2),$ while $r_1$ scales as $\sim 1/B$.

%We choose a real parameter $B\gg 0$ so that 
%\begin{equation}\label{Bparameter}
%    B^{-1}<\frac{r}{\sqrt{B}} \prod_{n=0}^{\infty}\left(1-\frac{\left(c(7\mu)^{\mu}(B\delta_1)\right)^{1/(\mu+1)}}{(B-1)^{n/(\mu+1)}}\right).
%\end{equation}

%To check that this is possible we apply Lemma \ref{calculus2} to the right hand side with $u=\left(c(7\mu)^{\mu}(B\delta_1)\right)^{1/(\mu+1)}$ and $\alpha=(B-1)^{1/(\mu+1)}$; note that $\delta_1 < c^{-1} (7\mu)^{-\mu}2^{-\mu-1}/B$ by construction, and so $u\in (0,0.5)$. We get 
%\[r_1 \prod_{n=0}^{\infty}\left(1-\frac{\left(c(7\mu)^{\mu}(B\delta_1)\right)^{1/(\mu+1)}}{(B-1)^{n/(\mu+1)}}\right) > r_1 e^{-\alpha/(\alpha-1)} = \frac{r}{\sqrt{B}} e^{-\frac{(B-1)^{1/(\mu+1)}}{(B-1)^{1/(\mu+1)}-1}}.  \]
%As $B$ approaches infinity the expression $$-\frac{(B-1)^{1/(\mu+1)}}{(B-1)^{1/(\mu+1)}-1}$$ stays bounded, and so the right hand side of the inequality will become larger than $B^{-1}$.

%By replacing $r$ with a smaller radius we can assume that $$\norm{\vec{f}-A\vec{x}}{r}<c^{-1}(7\mu)^{-\mu}2^{-\mu-1}.$$\daniel{need to make sure this is less than $1$} \borys{need to prove Schwarz' Lemma somewhere} Further shrinking $r$ to $r_1\colonequals r/\sqrt{B}$ we can assume that $$\delta_1\colonequals \norm{\vec{f}-A\vec{x}}{r_1}<c^{-1}(7\mu)^{-\mu}2^{-\mu-1}/B.$$

We will produce a sequence of constants $\eta_n, \delta_n, r_n$ and elements $\vec{\psi}_n\in \End^\op \Kr{\vec x}{r_{n+1}}, \vec{f}_n\in \End^\op\Kr{\vec x}{r_{n}}$, with $\vec\psi_n$ invertible, and such that $\vec{f}_1=\vec{f} \in \End^{\op}\Kr{\vec{x}}{r_1}$, $\delta_1=\norm{\vec{f}-A\vec{x}}{r_1}$, $\eta_1=1/3$. We will have $$\vec{f}_{n+1}=\vec{\psi}_n^{-1} \circ \vec{f}_n \circ  \vec{\psi}_{n},$$ $$r_n=(1-\eta_{n-1})r_{n-1},$$ $$\delta_n \colonequals \Vert \vec{f}_n - A\vec{x}\Vert_{r_n} < \delta_{n-1}/2,$$ $$\norm{\vec{\psi}_n-\vec{x}}{r_{n+1}}<c^{-1}\delta_n (1-\eta_n)(7\mu)^\mu \eta_n^{-\mu},$$

and $$\eta_n^{\mu+1}=\delta_n/(3^{\mu+1}\delta_1) \in (0,1).$$ Indeed,  iteratively apply Lemma \ref{Siegel iteration} to $\vec{f}_n, r_n, \eta_n$. For the lemma to be applicable we need to check two conditions: $\eta_n \in (0,1)$ and $c^{-1}(7\mu)^\mu \eta_n^{-\mu}\delta_n<r_n$. The first condition follows since $\delta_n<\delta_1$ and so $\eta_n = \left(\delta_n/(3^{\mu+1}\delta_1)\right)^{1/(\mu+1)} < 1/3.$ For the second condition, we have \[\eta_n^{-\mu} \delta_n = \left(\frac{\delta_n}{3^{\mu + 1}\delta_1}\right)^{-\mu/(\mu+1)} \delta_n=\delta_n^{1/(\mu+1)}(3^{\mu+1}\delta_1)^{\mu/(\mu+1)}<\delta_1^{1/(\mu+1)}3^\mu\delta_1^{\mu/(\mu+1)}=\delta_1 3^{\mu}.\] Thus it is enough to show \[c^{-1}(7\mu)^\mu\delta_1 3^{\mu} < r_n.\]
We estimate $r_n$ from below using the estimate $\delta_n < \delta_12^{-(n-1)}$
\begin{multline}\label{r-infinity-bound}
    r_n = r_1 \prod_{i=1}^{n-1} (1 - \eta_i)=r_1 \prod_{i=1}^{n-1}\left(1-\left(\frac{\delta_i}{3^{\mu + 1}\delta_1}\right)^{1/(\mu+1)}\right)>r_1 \prod_{i=1}^{n-1} \left(1-\frac{1}{3}2^{-(i-1)/\mu+1}\right)\\
    =r_1 \prod_{i=0}^{n-1}\left(1-\frac{1}{3}2^{-i/\mu+1}\right)>r_1 \prod_{i=0}^{\infty}\left(1-\frac{1}{3}2^{-i/\mu+1}\right) \overset{\text{by Lemma \ref{calculus2}}}{>}r_1 e^{-\frac{2^{1/(\mu+1)}}{2^{1/(\mu+1)}-1}}
\end{multline}
Therefore using inequality (\ref{Bparameter}) and the previous estimate we get \[c^{-1}(7\mu)^\mu \delta_1 3^{\mu}  <\frac{1}{2} r_1 e^{-\frac{2^{1/(\mu+1)}}{2^{1/(\mu+1)}-1}} < r_n . \]

%The second condition is equivalent to $\eta_n^\mu r_n > c (7\mu)^\mu \delta_n$. Since $\eta_n<1$, it suffices to show that $\eta_n^{\mu+1} r_n > c (7\mu)^\mu \delta_n$; from the formula for $\eta_n$ this is equivalent to $r_n>B^{-1}$. We have
%\begin{multline}\label{radius-bound}
 %  r_n=r_1\prod_{i=1}^{n-1}(1-\eta_i)=r_1 \prod_{i=1}^{n-1}\left(1-\left(c(7\mu)^{\mu} (B\delta_i\right))^{1/(\mu+1)}\right) \geqslant \\ r_1 \prod_{i=1}^{n-1}\left(1-\frac{ \left(c(7\mu)^{\mu} (B\delta_1\right))^{1/(\mu+1)}}{(B-1)^{(i-1)/(\mu+1)}}\right) > r_1 \prod_{i=0}^{\infty}\left(1-\frac{\left(c(7\mu)^{\mu} (B\delta_1\right))^{1/(\mu+1)}}{(B-1)^{i/(\mu+1)}}\right)\overset{\text{by }(\ref{Bparameter})}{>}B^{-1}. 
%\end{multline}

Thus, we can apply Lemma \ref{Siegel iteration} with $\vec{f}=\vec f_n, r=r_n, \eta=\eta_n$ to produce $\vec{\psi}_{n}$ and $\vec f_{n+1}=\vec{\psi}_n^{-1} \circ \vec{f}_n \circ \vec{\psi}_n$ such that 
$$\norm{\vec{\psi}_n-\vec{x}}{r_{n+1}}<c^{-1}\delta_n (1-\eta_n)(7\mu)^\mu \eta_n^{-\mu}$$ and 

$$\delta_{n+1} < c^{-1} (7\mu)^\mu \frac{\delta_n^2}{\eta_n^{\mu} r_n} = c^{-1}(7\mu)^\mu \delta_n^{1/(\mu+1)} \delta_1^{\mu/(\mu+1) }3^\mu r_n^{-1} \delta_n < c^{-1} (7\mu)^\mu \delta_1 3^\mu r_n^{-1} \delta_n $$

$$\overset{\text{By (\ref{r-infinity-bound})}}{<} c^{-1} (7\mu)^\mu \delta_1 3^\mu r_1^{-1} e^{\frac{2^{1/(\mu+1)}}{2^{1/(\mu+1)}-1}} \delta_n \overset{\text{By (\ref{Bparameter})}}{<} \frac{1}{2} \delta_n$$

Thus the infinite sequence $\vec{f}_n$ can be constructed as claimed.

%We get \[\delta_{n+1}\colonequals \norm{\vec{f}_n-A\vec{x}}{r_n(1-\eta_n)} \leqslant \frac{c(7\mu)^{\mu}}{\eta_n^{\mu}-c(7\mu)^{\mu}\delta_n} \delta_n^2 < \frac{c(7\mu)^{\mu}}{\eta_n^{\mu+1}-c(7\mu)^{\mu}\delta_n} \delta_n^2= \frac{\delta_n}{B-1}.\]\daniel{Should this be $\norm{\vec{f}_{n+1}-A\vec x}{r_n(1-\eta_n)}$? I'm also a bit worried about the stric inequality; why is $\eta_n^{\mu+1}-c(7\mu)^{\mu}\delta_n>0$?} 
The sequence $\delta_n$ converges to zero (at least) exponentially. Since $\eta_n=\frac{1}{3}(\delta_n/\delta_1)^{1/(\mu+1)}$, the product $\prod_n (1-\eta_n)$ converges. Let \[r_\infty\colonequals r_1 \prod_n(1-\eta_n).\] By Lemma \ref{fixed-Endop}(\ref{schwarz}), we have $\norm{\vec f_n-A\vec x}{r_\infty}<\delta_n$, and so the sequence of conjugates $\vec{f}_n$ converges to the function $A\vec{x}$ on the disk of radius $r_\infty$. 

We now show that the limit of $\vec{f}_n$ is also a conjugate of $\vec{f}$. Let $\vec{\Psi}_n=\vec{\psi}_1 \circ \dots \circ \vec{\psi}_n$. Then $\vec{\Psi}_n\in \End^\op \Kr{\vec x}{r_\infty}$ is invertible, and $\vec{\Psi}_n^{-1} \circ \vec{f} \circ \vec{\Psi}_n = \vec{f}_n$. By construction, \[\norm{\vec{\psi}_n-\vec{x}}{r_\infty}<c^{-1}\delta_n (1-\eta_n)(7\mu)^\mu \eta_n^{-\mu}= 3^{\mu+1}c^{-1}(7\mu)^\mu(1-\eta_n)\delta_1 \eta_n < 3^{\mu+1}c^{-1}(7\mu)^\mu \delta_1 \eta_n = Q \eta_n,\] where $Q>0$ does not depend on $n$. We have by Lemma \ref{Taylor} $\norm{\vec{\Psi}_{n+1}-\vec{\Psi}_n}{r_\infty}= \norm{\vec{\Psi}_n(\vec{\psi}_{n+1}(x))-\vec{\Psi}_n(x)}{r_\infty} < \frac{1}{r_\infty} \norm{\vec{\Psi}_n}{r_\infty} \norm{\vec{\psi}_{n+1}-\vec{x}}{r_\infty} <  \frac{Q}{r_\infty} \norm{\vec{\Psi}_n}{r_\infty}  \eta_{n+1}$. Since $\eta_n$ converges to zero, the sequence $\vec{\Psi}_n$ is Cauchy, and thus has a limit $\vec{\Psi}$. Since $\vec{f}_n = \vec{\Psi}_n^{ -1} \circ \vec{f} \circ \vec{\Psi}_n$, we have $\vec{\Psi}^{ -1} \circ \vec{f} \circ \vec{\Psi} = A\vec{x}$, using continuity of composition (Lemma \ref{fixed-Endop}(\ref{composition-inequality})).
\end{proof}

\begin{corollary}\label{cor:Siegel-app}
Suppose $F\in \End^{\op}\Kr{\vec{x}}{r}$ is a semisimple endomorphism. Suppose the eigenvalues $\lambda_i$ of the action of $F$ on $\mathscr{I}/\mathscr{I}^2$ are elements of $K$ with $|\lambda_i|\leqslant 1$ that satisfy Siegel's condition with parameter $\mu$. Then there exists a radius $r'<r$ and a collection of elements $y_1, \dots, y_n \in \Kr{\vec{x}}{r'}$ with the following two properties
\begin{enumerate}
    \item If $\lambda_1, \dots, \lambda_n \in \overline{\Q_\ell}$ are eigenvalues of $F$ on $\mathscr{I}/\mathscr{I}^2$ then $Fy_i=\lambda_i y_i$;
    \item For any integer $m$ the monomials in $y_i$ of degree at least $m$ have dense span in $\mathscr{I}_{r'}^m$.
    \end{enumerate}
\end{corollary}
\begin{proof}
After replacing $r$ with a smaller radius $\tilde{r}$ we can do a linear change of variables $\vec{x'}=M\vec{x}$ to make the action of $F$ on $\mathscr{I}_{\tilde{r}} / {\mathscr{I}_{\tilde{r}}}^2$ diagonal in the basis $\vec{x}'$. We can therefore assume that $F=A\vec{x}' + O(x^2)$, where $A$ is a diagonal matrix $A=\mathrm{diag}(\lambda_1, \dots, \lambda_n)$. We can now apply Theorem \ref{Siegel}: there exists a radius $r'<\tilde{r}$ and $\vec{\psi} \in \End^{\op}\Kr{\vec{x}'}{r'}$ such that $\vec{\psi}^{-1}\circ F \circ \vec{\psi} = A\vec{x}'$. Let $y_i=\psi_i^{-1}(\vec{x}')$, then $Fy_i=y_i(F(\vec{x}'))=A\vec{y}$. The span of the monomials in the  $x_i'$ of degree $m$ or larger is dense in $\mathscr{I}_{r'}^m$ by Lemma \ref{lemma:dense-monomials}. Since $\vec{x}'= \vec{\psi}(\vec{y})$, any monomial in $x_i'$ of degree $m$ is a (convergent) sum of monomials in $y_i$ of degree at least $m$. Hence monomials in $y_i$ of degree $m$ or larger have dense span in $\mathscr{I}_{r'}^m$.
\end{proof}

\section{Main theorem}                      
Having proven Theorem \ref{Siegel} and Corollary \ref{cor:Siegel-app}, we move on to the proof of Theorem \ref{main theorem}, which is a straightforward application.

\begin{proof}[Proof of Theorem \ref{main theorem} (Compare to {\cite[Proof of Theorem 1.2]{litt-inventiones})}]
We first observe that it suffices to consider the case where $X$ is affine and $k$ is finite. Indeed, we may take $X$ to be affine by deleting any closed point of $X$, which does not affect the hypotheses of the theorem. To see that we may reduce to the case where $k$ is finite, choose a finitely-generated integral $\mathbb{Z}$-algebra $R$ in which $\ell$ is invertible, a smooth proper $R$-curve $\overline{\mathscr{X}}$, a divisor $D$ in $\overline{\mathscr{X}}$ \'etale over $R$, and an isomorphism $\on{Frac}(R)\overset{\sim}{\to} k$ such that $(\overline{\mathscr{X}}\setminus D)_k$ is isomorphic to $X$. Now for any geometric point $\bar p$ lying over a closed point $p\in \on{Spec}(R)$ with residue field $k(p)$ of characteristic prime to $\ell$, the specialization map $$\pi_1^{\ell}(X_{\bar k})\to \pi_1^\ell((\overline{\mathscr{X}}\setminus D)_{\bar p})$$ is an isomorphism. Moreover, any semisimple arithmetic representation of $\pi_1^{\ell}(X_{\bar k})$ remains semisimple arithmetic when viewed as a representation of $\pi_1^\ell((\overline{\mathscr{X}}\setminus D)_{\bar p})$, by the argument of \cite[Proof of Theorem 1.1.3, Step 2]{litt-duke}. Thus it suffices to prove the theorem for $(\overline{\mathscr{X}}\setminus D)_p$, which is by construction a smooth affine curve over a finite field.

For the rest of the argument we assume $k$ is finite of characteristic different from $\ell$, and $X/k$ is a smooth affine curve. Let $\bar k$ be an algebraic closure of $k$. We may, after replacing $k$ with a finite extension, assume that $X$ has a $k$-rational point $x$; we let $\bar x$ be the geometric point obtained from $x$ via our choice of algebraic closure $\bar k$ of $k$. In this case $\pi_1^\ell(X_{\bar k}, \bar x)$ is a free pro-$\ell$ group, and hence $\mathbb{Z}_\ell\Laur{\pi_1^\ell(X_{\bar k}, \bar x)}$ is (non-canonically) isomorphic to a noncommutative power series ring over $\mathbb{Z}_\ell$; fix such an isomorphism. As $\bar x$ was obtained from a rational point of $X$, the absolute Galois group of $k$ acts naturally on $\mathbb{Z}_\ell\Laur{\pi_1^\ell(X_{\bar k}, \bar x)}$.

Let $F$ be the Frobenius element in the absolute Galois group of $k$. Consider the action of $F$ on $\mathbb{Q}_\ell\Laur{\pi_1^{\ell}(X_{\bar k}, \bar{x})}$. By \cite[Theorem 2.20]{litt-inventiones} the action of $F$ is semisimple in the sense of Definition \ref{def-semisimple}. By the Weil conjectures for curves, the eigenvalues $\lambda_1, \dots, \lambda_n$ of the action of $F$ on $\mathscr{I}/\mathscr{I}^2 = H^1(X_{\bar{k}}, \overline{\Q_\ell})^\vee$ \cite[Proposition 2.4]{litt-inventiones} are $q$-Weil numbers of weights $-1$ and $-2$. In particular they are algebraic and therefore by Proposition \ref{prop:baker} they satisfy Siegel's condition for some parameters $c,\mu$. Let $K/\Q_\ell$ be a finite extension that contains all $\lambda_i$. By Theorem \ref{Siegel} and Corollary \ref{cor:Siegel-app} there exists a radius $r$ such that the ideal $\mathscr{I}_r^n\subset\Kr{\pi_1^{\et}(X_{\kbar}, \bar{x})}{r}$ is (topologically) spanned by $F$-eigenvectors for all $n\geqslant 0$. These eigenvectors are monomials in the elements $y_1, \cdots, y_n\in \mathscr{I}_r$ provided by Corollary \ref{cor:Siegel-app}, where the eigenvalue $\lambda_i$ corresponding to $y_i$ also appears as an eigenvalue of the Frobenius action on $H^1(X_{\bar k}, K)^\vee$, and hence is a $q$-Weil number of weight $-1$ or $-2$.  In particular $\mathscr{I}_r^n$ is topologically spanned by $F$-eigenvectors whose corresponding eigenvalues are $q$-Weil numbers of weight $\leqslant-n$.

Let $N(X,\ell)$ be an arbitrary real number strictly larger than $-\log(r)/\log(\ell) \in \R.$  Now suppose $\rho$ is a semisimple arithmetic representation trivial modulo $\ell^{N(X, \ell)}$. There exists a finite extension $k'/k$ such that $\rho$ extends to a representation $\rho': \pi_1^{\et}(X_{k'}, \bar{x}) \to \GL_n(\overline{\Z_\ell})$ of the arithmetic fundamental group of $X_{k'}$. There exists an integer $m$ such that $F^m$ lifts to an element of $\pi_1^{\et}(X_{k'}, \bar{x})$. Let $A \in \GL_n(\overline{\Z_\ell})$ be the matrix $A\colonequals \rho'(F^m)$. Since the representation $\rho$ is trivial modulo $\ell^{N(X, \ell)}>1/r$ it extends to a continuous representation $$\hat{\rho}:\Kr{\pi_1^{\ell}(X_{\kbar}, \bar{x})}{r} \to \on{Mat}_{n\times n}(\overline{\Q_\ell})$$ by Lemma \ref{extends-to-Kr}. Moreover, by arithmeticity, $\hat{\rho}(F^m(g ))=A\hat{\rho}(g)A^{-1}$ for every $g \in \Kr{\pi_1^{\et}(X_{\kbar}, \bar{x})}{r}.$ Let $w'$ denote the the most negative weight of an eigenvalue of the conjugation action of $A$ on $\mathrm{Mat}_{n \times n}(\overline{\Q_\ell}),$ if any such exist, and $0$ otherwise. Let $w=\max(-w', 0).$ Every monomial $Y \in \mathscr{I}_r^{w+1}$ in the $y_i$ satisfies $A\hat\rho(Y)A^{-1}=u \hat\rho(Y)$ for a $q$-Weil number $u$ of weight less than $-w$. Since no such numbers are eigenvalues of the conjugation action of $A$ on $\mathrm{Mat}_{n \times n}(\overline{\Q_\ell}),$ the image of every monomial in $\mathscr{I}_r^{w+1}$ under $\hat{\rho}$ is zero. As such monomials topologically span $\mathscr{I}_r^{w+1}$ by Corollary \ref{cor:Siegel-app}, we have $\hat{\rho}(\mathscr{I}_r^{w+1})=0$, and hence that $\rho(\mathscr{I}_r^{w+1}\cap \mathbb{Z}_\ell\Laur{\pi_1^\ell(X_{\bar k}, \bar x)})=0$. But $\mathscr{I}_r^{w+1}\cap \mathbb{Z}_\ell\Laur{\pi_1^\ell(X_{\bar k}, \bar x)}$ is $\mathscr{I}^{w+1}$, where $\mathscr{I}$ is the augmentation ideal of $\mathbb{Z}_\ell\Laur{\pi_1^\ell(X_{\bar k}, \bar x)}$. Therefore, $\rho$ is unipotent. But a unipotent semisimple representation is trivial.
\end{proof}
\begin{remark}\label{rmk:proper-diagonalization}
In the course of the proof, we show that for $X$ a smooth affine curve over a finite field $k$, and $\ell$ a prime different from the characteristic of $k$, there exists a finite extension $K$ of $\mathbb{Q}_\ell$ and an $r>0$ such that the Banach algebra $\Kr{\pi_1^{\ell}(X_{\bar k}, \bar x)}{r}$ is topologically spanned by Frobenius eigenvectors. In fact the same statement for smooth proper curves follows immediately, as if $X$ is smooth and proper and $y\in X$ is a closed point, the map $\pi_1^{\ell}((X\setminus y)_{\bar k})\to \pi_1^\ell(X_{\bar k})$ is surjective.
\end{remark}
\begin{proof}[Proof of Corollary \ref{cor:abelian-variety}]
As in the statement, we let $X/k$ be a curve over a finitely-generated field, and $A/X_{\bar k}$ an Abelian scheme. Suppose that $A_\eta$ had full $\ell^M$-torsion. Then the natural geometric monodromy representation $$\pi_1^{\et}(X_{\bar k}, \bar x)\to \GL(T_\ell(A_{\bar x}))$$ is trivial mod $\ell^M$ for any geometric point $\bar x$ of $X_{\bar k}$. As this representation is semisimple arithmetic (as are all representations arising from geometry---semisimplicity follows from \cite[3.4.1(iii)]{weilii}, and arithmeticity by spreading out), Theorem \ref{main theorem} implies that it is trivial, and in particular every $\ell$-power torsion point of $A_\eta$ is rational. Thus by the Lang-N\'{e}ron theorem \cite[Theorem 2.1]{lang-neron}, the natural map $$\tau:\on{Tr}_{\eta/\bar k}(A_\eta)_\eta\to A_\eta$$ had image containing all the $\ell$-power torsion points of $A_\eta$. As the $\ell$-power torsion is Zariski-dense, this implies that $\tau$ is surjective and hence an isogeny for dimension reasons, proving the statement.
\end{proof}
\section{Remarks and extensions}\label{sec:remarks}
\subsection{A suggestive correspondence}
Let $X$ be a smooth proper curve over a finite field $k$, and $\ell$ a prime different from the characteristic of $k$. Let $x\in X(k)$ be a rational point, and $\bar x$ the geometric point of $X$ associated to $x$ by a choice of algebraic closure of $k$. Let $N=N(X,\ell)$ be as in Theorem \ref{main theorem}. One consequence of Theorem \ref{Siegel} is a Galois-equivariant description of the category of lisse $\overline{\mathbb{Q}_\ell}$-sheaves on $X_{\bar k}$ admitting lattices which are trivial mod $\ell^N$ in terms of linear algebra data. We view this as a (very weak) $\ell$-adic analogue of non-abelian Hodge theory.
\begin{definition}
Let $\mathscr{H}_\ell(X)$ be the category whose objects consist of pairs $$(V, \theta: V\to V\otimes H^1(X_{\bar k}, \overline{\mathbb{Q}_\ell})),$$ where $V$ is a finite-dimensional $\overline{\mathbb{Q}_\ell}$-vector space and $\theta$ is a linear map. A morphism between $(V, \theta)$ and $(V', \theta')$ is a linear map $f: V\to V'$ so that the diagram
$$\xymatrix{
V \ar[r]^-\theta \ar[d]^f & V\otimes H^1(X_{\bar k}, \overline{\mathbb{Q}_\ell})\ar[d]^{f\otimes \on{id}}\\
V' \ar[r]^-{\theta'} & V'\otimes H^1(X_{\bar k}, \overline{\mathbb{Q}_\ell})
}$$
commutes.
\end{definition}
Let $\on{Sh}_{\ell, N}(X_{\bar k})$ be the full subcategory of the category of continuous representations  $$\rho: \pi_1^{\ell}(X_{\bar k}, \bar x)\to \GL(V),$$ where $V$ is a finite-dimensional $\overline{\mathbb{Q}_\ell}$-vector space, such that there exists a $\overline{\mathbb{Z}_\ell}$-sublattice $W$ of $V$, stable under the action of $\pi_1^{\ell}(X_{\bar k}, \bar x)$, and such that $\pi_1^{\ell}(X_{\bar k}, \bar x)$ acts on $W/(\ell^N)W$ trivially. We now construct a functor $$H: \text{Sh}_{\ell, N}(X_{\bar k})\to \mathscr{H}_\ell(X).$$ Let $K, r$ be as in the proof of Theorem \ref{main theorem} and Remark \ref{rmk:proper-diagonalization}, so that Frobenius acts diagonalizably on $\Kr{\pi_1^{\ell}(X_{\kbar}, \bar{x})}{r}$. Letting $\mathscr{I}_r\subset \Kr{\pi_1^{\ell}(X_{\kbar}, \bar{x})}{r}$ be the augmentation ideal, note that the natural map $$\mathscr{I}_r\to \mathscr{I}_r/\mathscr{I}_r^2\simeq H^1(X_{\bar k}, K)^\vee$$ admits a unique Frobenius-equivariant splitting, given by the span of the weight $-1$ Frobenius-eigenvectors. Thus the span of the weight $-1$ eigenvectors yields a copy of $H^1(X_{\bar k}, K)^\vee$ inside of $\Kr{\pi_1^\ell(X_{\bar k}, \bar x)}{r}$. Now let $V$ be an object of $\text{Sh}_{\ell, N}(X_{\bar k})$. By Lemma \ref{extends-to-Kr}, we have a natural action of $\Kr{\pi_1^{\ell}(X_{\kbar}, \bar{x})}{r}$ on $V$, and thus viewing $H^1(X_{\bar k}, K)^\vee$ as subspace of $\Kr{\pi_1^{\ell}(X_{\kbar}, \bar{x})}{r},$ we obtain a natural map $$H^1(X_{\bar k}, K)^\vee\otimes_K V\to V.$$ By adjointness we thus obtain an object $$V\to V\otimes H^1(X_{\bar k}, \overline{\mathbb{Q}_\ell})$$ of $\mathscr{H}_\ell(X)$. 

This construction is evidently functorial. One can verify from the definition that the functor $H$ is fully faithful. Moreover there is a natural Frobenius action on the set of isomorphism classes of objects of $\mathscr{H}_\ell(X)$ (via the action of Frobenius on $H^1(X_{\bar k}, \overline{\mathbb{Q}_\ell})$), and $H$ induces a Frobenius-equivariant map from isomorphism classes of objects of $\text{Sh}_{\ell, N}(X_{\bar k})$ to isomorphism classes of objects of $\mathscr{H}_\ell(X)$. We can interpret Theorem \ref{main theorem} as the full faithfulness of this functor, combined with the fact that any object of $\mathscr{H}_\ell(X)$, fixed up to isomorphism by the action of Frobenius, is \emph{nilpotent}, in the sense that for $n\gg 0$, the composition $$\theta^n: V\to V\otimes H^1(X_{\bar k}. \overline{\mathbb{Q}_\ell})^{\otimes n}$$ is zero.

Using the semisimplicity of the Frobenius action on $\Qlaur{\pi_1^{\et}(X_{\kbar}, \bar{x})}$ \cite[Theorem 2.20]{litt-inventiones}, one may extend this correspondence to the case of non-proper $X$, though doing so seems to depend on some choices. It would of course be very interesting to find a variant of this construction for residually nontrivial representations.
\subsection{Residually nontrivial representations}
It is natural to ask if results similar to Theorem \ref{main theorem} hold for residually nontrivial arithmetic representations. Indeed, a version of Theorem \ref{Siegel} in the commutative setting (that is, a mild generalization of the main result of~\cite[Section 4]{herman-yoccoz}, allowing ``resonance"), with an identical proof, implies:
\begin{theorem}\label{thm:def-rings-version}
Let $X$ be a smooth curve over a finite field $k$, $\bar x$ a geometric point of $X$, and $$\overline{\rho}: \pi_1^{\et}(X, \bar x)\to \GL_n(\mathbb{F}_{\ell^r})$$ a representation which is absolutely irreducible when restricted to $\pi_1^{\et}(X_{\bar k}, \bar x)$, with $\ell$ different from the characteristic of $k$. Let $R_{\overline{\rho}}$ be the deformation ring of $\overline{\rho}|_{\pi_1^{\et}(X_{\bar k}, \bar x)}$, and let $U_{\overline{\rho}}$ be its rigid generic fiber. Let $K$ be an $\ell$-adic field with residue field $\mathbb{F}_{\ell^r}$, and let $$\rho: \pi_1^{\et}(X, \bar x)\to \GL_n(\mathscr{O}_K)$$ be a continuous lift of $\overline{\rho}$; let $[\rho]\in U_{\overline{\rho}}$ be the point corresponding to $\rho|_{\pi_1^{\et}(X_{\bar k}, \bar x)}$. If the action of Frobenius on $$H^1(X_{\bar k}, \rho\otimes\rho^\vee)$$ is semisimple, there exists an open neighborhood $V$ of $[\rho]$ in $U_{\overline{\rho}}$ such that the Frobenius action on $V$ is conjugate to a linear map.
\end{theorem}
In dynamics, a neighborhood $V$ as above is referred to as a \slantsf{Siegel disk}. Note that the hypothesis on the semisimplicity of the Frobenius action on $H^1(X_{\bar k}, \rho\otimes\rho^\vee)$ would follow for all $\rho$ from the Tate conjecture, by L.~Lafforgue's work on the Langlands program: \cite[Corollaire VII.8]{lafforgue} implies that $\rho\otimes \rho^\vee$ ``arises from geometry," whence the Tate conjecture would imply that the Frobenius action on its cohomology groups is semisimple.

As a corollary of Theorem \ref{thm:def-rings-version}, one obtains that for $\rho$ as in the theorem statement, there exists a neighborhood of $[\rho]$ containing no Frobenius-periodic points aside from $[\rho]$ (that is, no arithmetic representations). This is proven unconditionally in \cite[Theorem 1.1.3]{litt-duke}. That said, it would in our view be quite interesting to understand Siegel disks in $U_{\bar\rho}$; for example, if $U_{\bar\rho}$ was covered by Siegel disks for iterates of Frobenius, the Hard Lefschetz theorem would follow for all lifts of $\bar\rho$, by the strategy of \cite{esnault-kerz}.

\begin{proof}[Sketch proof of Theorem \ref{thm:def-rings-version}]
The proof of Theorem \ref{Siegel} works verbatim in the commutative setting, giving the following result. Let $K$ be an $\ell$-adic field, $R$ a Tate algebra over $K$, and $F: R\to R$ a continuous endomorphism. Suppose $F$ preserves a maximal ideal $\mathfrak{m}$ of $R$, and acts semisimply on the completion $\widehat{R}$ of $R$ at $\mathfrak{m}$ (i.e.~$F$ acts semisimply on the finite-dimensional $K$-vector spaces $R/\mathfrak{m}^n$ for all $n$). Suppose moreover that the action of $F$ on $\mathfrak{m}/\mathfrak{m}^2$ has eigenvalues satisfyiing $\ell$-Siegel's condition with parameters $c,\mu$ for some $c,\mu>0$. Then there exists an affinoid neighborhood of $[\mathfrak{m}]\in \on{Sp}(R)$ on which $F$ is conjugate to a linear map.

We now choose a Frobenius-stable open ball $U$ containing $[\rho]$ in the rigid generic fiber of $R_{\bar\rho}$; as $R_{\bar\rho}$ is a power series ring over $W(k)$ by the absolute irreducibility of $\bar\rho$, we may choose $U$ to be the spectrum of a Tate algebra $R$. Thus it is enough to check the hypotheses of the result of the previous paragraph, taking $F$ to be the Frobenius automorphism of $R$ and $\mathfrak{m}$ to be the maximal ideal corresponding to $\rho$. The semisimplicity hypothesis follows from the assumption of the semisimplicity of the Frobenius action on $H^1(X_{\bar k}, \rho\otimes\rho^\vee)=(\mathfrak{m}/\mathfrak{m}^2)^\vee$ by an argument identical to the proof of \cite[Theorem 5.1.8]{litt-duke}. And  \cite[Corollaire VII.8]{lafforgue} implies that the eigenvalues of the Frobenius action on $\mathfrak{m}/\mathfrak{m}^2$ are Weil numbers, hence algebraic; thus they satisfy $\ell$-Siegel's condition for some $\mu$ by Proposition \ref{prop:baker}.
\end{proof}
\begin{question}
Let $U_{\bar\rho}$ be as in the statement of Theorem \ref{thm:def-rings-version}. Is $U_{\bar\rho}$ covered by Siegel disks for iterates of Frobenius? That is, for each point $\nu$ of $U_{\bar\rho}$, does there exist a finite extension $k'$ of $k$, a representation $\rho_\nu: \pi_1^{\et}(X_{k'}, \bar x)\to \GL_n(\mathscr{O}_L)$ lifting $\bar\rho$ (with $L$ an $\ell$-adic field with residue field $\mathbb{F}_{\ell^r}$), and a neighborhood $V$ of $[\rho_\nu]$ containing $\nu$, such that the action of the Frobenius of $k'$ on $V$ is conjugate to a linear map?
\end{question}
We view our proof of Theorem \ref{main theorem} as showing that the trivial representation is contained in a ``noncommutative Siegel disk."
\subsection{Other questions}
Theorem \ref{main theorem} is unsatisfying in a number of ways. First, it is natural to ask if the constant $N(X, \ell)$ appearing in the statement of the theorem may be taken to tend to zero as $\ell\to \infty$; as remarked in the introduction, this is known in characteristic zero by \cite[Theorem 1.1.13]{litt-duke}. It is not clear to us if there is a plausible improvement of the $\ell$-adic form of Baker's theorem on linear forms in logarithms which would, by  our method, imply that one may take $N(X,\ell)\to 0$ as $\ell\to\infty$, but the following would imply that one may take $N(X,\ell)$ to be independent of $\ell$:
\begin{question}\label{baker-question}
Let $\lambda_1,\cdots, \lambda_n\in \overline{\mathbb{Q}}$ be a collection of algebraic numbers, stable under the action of $\on{Gal}(\overline{\mathbb{Q}}/\mathbb{Q})$. For each prime $\ell$, the $\ell$-adic form of Baker's theorem yields constants $c_\ell, \mu_\ell$ such that the $\lambda_i$ satisfy $\ell$-Siegel's condition with parameters $c_\ell,\mu_\ell$, as in Definition \ref{def:l-siegel}. Is it true that the constants $c_\ell, \mu_\ell$ may be taken to be  independent of $\ell$?
\end{question}

It is also natural to ask for an analogue of Theorem \ref{main theorem} with better uniformity in e.g.~the genus or gonality of the curve $X$, or one which depends only on the function field $\bar k(X)$. For example, the geometric torsion conjecture \cite{bakker-tsimerman2} predicts that in the case that $\rho$ arises from the $\ell$-adic Tate module of a traceless Abelian scheme $A$, there is a bound on the torsion $$(\rho\otimes \mathbb{Q}_\ell/\mathbb{Z}_\ell)^{\pi_1^{\ell}(X_{\bar k}, \bar x)}$$ depending only on the gonality of $X$ and the rank of $\rho$ (here $(\rho\otimes \mathbb{Q}_\ell/\mathbb{Z}_\ell)^{\pi_1^{\ell}(X_{\bar k}, \bar x)}$ is naturally isomorphic to the rational $\ell$-power torsion of the generic fiber of $A$). A weaker question, plausibly approachable via the methods of this paper, is to ask for a bound on $N$ depending only on the function field $k(X)$ and not on $X$ itself.
\bibliographystyle{alpha}
\bibliography{arithmetic-reps-3}

\begin{thebibliography}{Nad89}

\bibitem[BT16]{bakker-tsimerman}
Benjamin Bakker and Jacob Tsimerman.
\newblock {$p$}-torsion monodromy representations of elliptic curves over
  geometric function fields.
\newblock {\em Ann. of Math. (2)}, 184(3):709--744, 2016.

\bibitem[BT18]{bakker-tsimerman2}
Benjamin Bakker and Jacob Tsimerman.
\newblock The geometric torsion conjecture for abelian varieties with real
  multiplication.
\newblock {\em J. Differential Geom.}, 109(3):379--409, 2018.

\bibitem[Con06]{lang-neron}
Brian Conrad.
\newblock Chow's {$K/k$}-image and {$K/k$}-trace, and the {L}ang-{N}\'{e}ron
  theorem.
\newblock {\em Enseign. Math. (2)}, 52(1-2):37--108, 2006.

\bibitem[Del80]{weilii}
Pierre Deligne.
\newblock La conjecture de {W}eil: {II}.
\newblock {\em Publications Math{\'e}matiques de l'Institut des Hautes Etudes
  Scientifiques}, 52(1):137--252, 1980.

\bibitem[EK21]{esnault-kerz}
H{\'e}l{\`e}ne Esnault and Moritz Kerz.
\newblock {\'E}tale cohomology of rank one $\ell $-adic local systems in
  positive characteristic.
\newblock {\em Selecta Mathematica}, 27(4):1--25, 2021.

\bibitem[HT06]{hwang-to}
Jun-Muk Hwang and Wing-Keung To.
\newblock Uniform boundedness of level structures on abelian varieties over
  complex function fields.
\newblock {\em Math. Ann.}, 335(2):363--377, 2006.

\bibitem[HY83]{herman-yoccoz}
Michael Herman and Jean-Christophe Yoccoz.
\newblock Generalizations of some theorems of small divisors to non archimedean
  fields.
\newblock In {\em Geometric dynamics}, pages 408--447. Springer, 1983.

\bibitem[Laf02]{lafforgue}
Laurent Lafforgue.
\newblock Chtoucas de {D}rinfeld et correspondance de {L}anglands.
\newblock {\em Invent. Math.}, 147(1):1--241, 2002.

\bibitem[Lit18]{litt-inventiones}
Daniel Litt.
\newblock Arithmetic representations of fundamental groups {I}.
\newblock {\em Invent. Math.}, 214(2):605--639, 2018.

\bibitem[Lit21]{litt-duke}
Daniel Litt.
\newblock {Arithmetic representations of fundamental groups, II: Finiteness}.
\newblock {\em Duke Mathematical Journal}, pages 1 -- 47, 2021.

\bibitem[MB81]{moret1981familles}
Laurent Moret-Bailly.
\newblock Familles de courbes et de vari{\'e}t{\'e}s ab{\'e}liennes sur
  $\mathbb{P}^1$.
\newblock {\em S{\'e}m. sur les pinceaux de courbes de genre au moins deux (ed.
  L. Szpiro). Ast{\'e}riques}, 86:109--140, 1981.

\bibitem[Nad89]{nadel}
Alan~Michael Nadel.
\newblock The nonexistence of certain level structures on abelian varieties
  over complex function fields.
\newblock {\em Ann. of Math. (2)}, 129(1):161--178, 1989.

\bibitem[R{\"{u}}s72]{riissmann}
H~R{\"{u}}ssmann.
\newblock Kleine {N}enner ii: {B}emerkungen zur {N}ewton'schen {M}ethode.
\newblock {\em Nachr. Akad. Wiss. G{\"o}ttingen Math.-Phys. KI}, pages 1--10,
  1972.

\bibitem[Yu94]{yu-baker}
Kunrui Yu.
\newblock Linear forms in {$p$}-adic logarithms. {III}.
\newblock {\em Compositio Math.}, 91(3):241--276, 1994.

\bibitem[Zeh77]{zehnder}
Eduward Zehnder.
\newblock A simple proof of a generalization of a theorem by {C.L.} {S}iegel.
\newblock In {\em Geometry and topology}, pages 855--866. Springer, 1977.

\end{thebibliography}

\end{document}